\documentclass[12pt,a4paper]{amsart}
\usepackage[colorlinks=true,         
            linkcolor=blue,
            citecolor=red,
            urlcolor=blue]{hyperref}
\usepackage{a4wide, amssymb,
amsmath, hyperref,mathrsfs,latexsym,stmaryrd,
mathrsfs,amsaddr,cancel,ifthen}
\usepackage{tikz}
\usepackage{tikz-cd}
\usetikzlibrary{matrix,arrows,decorations.pathmorphing}
\DeclareMathAlphabet{\mathpzc}{OT1}{pzc}{m}{it}

\newtheorem{thm}{Theorem}[section] 
\newtheorem{corollary}[thm]{Corollary}
\newtheorem{lemma}[thm]{Lemma} 
\newtheorem{Prop}[thm]{Proposition}
\newtheorem{conj}[thm]{Conjecture}

\theoremstyle{definition}

\theoremstyle{remark}
\newtheorem{remqee}[thm]{Remark}

\numberwithin{equation}{section}

\title[Semistable Higgs bundles on elliptic surfaces]{Semistable Higgs bundles on elliptic surfaces}

\author{Ugo Bruzzo$^{\P\S}$ and Vitantonio Peragine$^{\P}$}

\address{\small $^\P$Scuola Internazionale Superiore di Studi Avanzati (SISSA), \\
Via Bonomea 265, 34136 Trieste, Italia\\
$^\S$Departamento de Matem\'atica, Universidade Federal da Para\'iba,  \\ Campus I, Jo\~ao Pessoa, PB, Brazil\\
$^\S$INFN (Istituto Nazionale di Fisica Nucleare), Sezione di Trieste\\
$^\S$IGAP (Institute for Geometry and Physics), Trieste\\
$^\S$Arnold-Regge Center for Algebra, Geometry   and Theoretical Physics, Torino}

\date{19 August 2020; revised 12 November 2020 and 3 January 2021} 
\thanks{U.B.'s research is partly supported by PRIN ``Geometria delle variet\`a algebriche" {and INdAM-GNSAGA. He 
  is} a member of the VBAC group.}

\begin{document}

\begin{abstract}
We analyze Higgs bundles $(V,\phi)$ on a class of elliptic surfaces $\pi:X\to B$, whose underlying vector bundle $V$ has vertical determinant and is fiberwise semistable. We prove that if the spectral curve of $V$ is reduced, then the Higgs field $\phi$ is \emph{vertical}, while if the bundle $V$ is fiberwise regular with reduced (resp., integral) spectral curve, and if its rank and second Chern number satisfy an inequality involving the genus of $B$ and the degree of the fundamental line bundle of $\pi$ (resp., if the fundamental line bundle is sufficiently ample), then $\phi$ is \emph{scalar}. We apply these results to the problem of characterizing slope-semistable Higgs bundles with vanishing discriminant on the class of elliptic surfaces considered, in terms of the semistability of their pull-backs via maps from arbitrary (smooth, irreducible, complete) curves to $X$.
\end{abstract}

\maketitle 
\setcounter{tocdepth}{1}
\tableofcontents 
\thispagestyle{empty}
\section{Introduction}
One of the aims of the present paper is to study a conjecture about semistable Higgs bundles on elliptic surfaces; in doing so, we also find some structural results about such bundles, which may be of independent interest.

Let $Y$ be a non-singular, projective scheme of dimension $n\geqslant 2$, defined over an algebraically closed field $k$ of characteristic $0$, and let $F$ be a coherent sheaf on $Y$, with rank $r$ and Chern classes $\mathrm{c}_i$. The \emph{discriminant} of $F$ is the characteristic class
$$
\Delta(F):=2r\mathrm{c}_2-(r-1)\mathrm{c}_1^2.
$$
Fix a polarization $H$ on $Y$. The Bogomolov inequality asserts that, if the sheaf $F$ is torsion-free and slope-semistable with respect to the polarization $H$, then its discriminant satisfies
\begin{equation}
\label{bog ineq}
\Delta(F)\cdot  H^{n-2}\geqslant 0.
\end{equation}
This was first proved in \cite{bog}
for locally free sheaves in the case $k=\mathbb{C},n=2$; see  \cite{langer}  for the general case. It is then natural to look for some sort of characterization of slope-semistable sheaves for which the lower bound $0$ for  the left hand side of \eqref{bog ineq} is attained. The main result in this direction is the following theorem, first proved in \cite{nakayama}, and then, independently and with a different proof, in \cite{bruzzo herna} (there the condition $\Delta(F)\cdot H^{n-2}=0$ is replaced by the stronger $\Delta(F)=0$ {in $ H^4(Y,\mathbb Q)$}; assuming $k=\mathbb{C}$, the two {conditions} are in fact equivalent, as soon as $F$ is locally free and slope-semistable with respect to the polarization $H$, as can be proved using Theorem 2 of \cite{simp 2}):

\begin{thm}
\label{caratt}
Let $(Y,H)$ be a \emph{complex} polarized variety. Then, for a locally free sheaf $F$ on $Y$, the following are equivalent:
\begin{enumerate}
\item[(i)] $F$ is slope-semistable and {$\Delta(F) =0$ in $ H^4(Y,\mathbb Q)$;}
\item[(ii)] for each pair $(C,f)$, where $C$ is an irreducible, non-singular, projective curve, and $f:C\to Y$ a morphism, the pull-back $f^\ast F$ of $F$ to $C$ along $f$ is semistable.
\end{enumerate}
\end{thm}

One also has a version of Bogomolov inequality for Higgs sheaves: again, let $(Y,H)$ be a polarized variety of dimension $n\geqslant 2$, defined over an algebraically closed field $k$ of characteristic $0$. Then any torsion-free, slope-semistable Higgs sheaf $(F,\phi)$ on $(Y,H)$ satisfies $\Delta(F)\cdot  H^{n-2}\geqslant 0$. It was first proved by Simpson in \cite{simp 1} in the case $k=\mathbb{C}$ (and   for stable holomorphic Higgs bundles on compact K\"ahler manifolds), using his \emph{generalized Kobayashi-Hitchin correspondence};\footnote{Simpson indeed proved the inequality only for stable Higgs bundles; the semistable case follows from the generalized (approximate) Hitchin-Kobayashi correspondence for \emph{semistable} Higgs bundles \cite{62, li}.} see again \cite{langer} for the general case. It is then somehow natural to expect the following analogue of Theorem \ref{caratt} to hold \cite{bruzzo herna, bruzzo beatriz}:

\begin{conj}
\label{conjec}
Let $(Y,H)$ be a \emph{complex} polarized variety, and let $(F,\phi)$ be a Higgs bundle on $Y$. Then the following are equivalent:
\begin{enumerate}
\item[(i)] $(F,\phi)$ is semistable with vanishing discriminant;
\item[(ii)] for each morphism $f:C\to Y$, with $C$ and $f$ as in Theorem \ref{caratt}, the pull-back $f^\ast(F,\phi)$
 is semistable.
\end{enumerate}
\end{conj}

The implication  (i) $\Rightarrow$ (ii) of \ref{conjec} was proved in \cite{bruzzo herna, bruzzo beatriz}. Moreover, by the Higgs version of Mehta-Ramanathan theorem \cite{simp 2}, any Higgs bundle satisfying (ii)  of {Conjecture} \ref{conjec} is semistable. So, what is left to be proved in order to establish the validity of the conjecture, is the statement that a Higgs bundle with non-zero discriminant is unstable when pulled back to a suitable curve. We remark that Theorem \ref{caratt} is true, more generally, for reflexive sheaves \cite{nakayama}. Thus one may formulate a \emph{reflexive version} of Conjecture \ref{conjec}. {However}, in this paper we shall work with smooth surfaces, where there is no distinction between vector bundles and reflexive sheaves, and so we will stick to the locally free case.

Conjecture \ref{conjec} is by now known to be true for several classes of varieties, including those with nef tangent bundle \cite{bruzzo alessio}, K3 surfaces \cite{bruzzo valeriano}, and, more generally, Calabi-Yau varieties \cite{bruzzo capasso}. In this paper we study Conjecture \ref{conjec} in the case of elliptic surfaces. 
We restrict ourselves to the case of non-isotrivial Weierstrass fibrations $\pi:X\to B$ without cuspidal singular fibers. Generalizations (e.g., to elliptic surfaces without sections, or with multiple fibers) will be treated elsewhere. Using, among other things, the equivalence of ordinary and Higgs-semistability for Higgs bundles on curves which are either rational or elliptic \cite{bruzzo alessio,garcia-prada}, we remark (Proposition \ref{Prop 56}) that it is enough to prove the conjecture for Higgs bundles $(V,\phi)$ on $X$ whose underlying vector bundle $V$ has vertical (or even trivial) determinant, and has semistable restriction to the closed fibers of $\pi$. This allows us to focus our attention on Higgs bundles $(V,\phi)$ on $X$, with $V$ assumed to have vertical determinant, and to be fiberwise semistable. As shown by Morgan, Friedman and Witten in \cite{mfw}, it is possible to associate to such a $V$ an effective divisor  $C_V$ on $X$, called the \emph{spectral curve} of $V$, belonging to the the linear system $\left|r\Sigma+\pi^\ast\mu\right|$. Here $\Sigma$ is the identity section of $\pi:X\to B$, $r$ is the rank of $V$, and $\mu$ is a suitable line bundle on $B$, whose degree equals the second Chern number of $V$. In the fiberwise regular case, the degree $r$ cover $C_V\hookrightarrow X\xrightarrow{\pi}B$ determines (the isomorphism class of) $V$, up to the choice of an invertible sheaf on $C_V$.

The cotangent bundle of $\Omega_X$ possesses a distinguished invertible subsheaf, i.e., the pull-back along $\pi$ of the canonical line bundle of the base curve $B$. We call Higgs fields on $V$ factoring through the inclusion $V\otimes\pi^\ast\omega_B\hookrightarrow V\otimes\Omega_X$  \emph{vertical}; by analyzing the restrictions of $V$ and $\Omega_X$ to a general closed fiber of $\pi$,
we show (Proposition \ref{propo} and Corollary \ref{corollll}) that, when the spectral curve of $V$ is reduced, these are the only Higgs fields which $V$ supports.

Every vector bundle on a smooth variety has a canonical family of Higgs fields on it, which we call \emph{scalar}, parametrized by the space of global $1$-forms on the variety. Using a Lemma on the relative incidence correspondence of divisors of the form $r\Sigma$ on the total space of our fibration $\pi:X\to B$ (Lemma \ref{il lemma bla bla bla}), we show (Proposition \ref{uno dei risultati}) that if $V$ is fiberwise regular, with reduced (resp., integral) spectral curve, and if the rank $r\geqslant 2$ and second Chern number $e$ of $V$ satisfy the inequality $e\geqslant (r-1)d+2g-1$ (resp., if $d\geqslant 2g-1$), where $d\geqslant 1$ is the degree of the fundamental line bundle of $\pi$, and $g$ the genus of $B$, then every Higgs field on $V$ is necessarily scalar; in particular, for every Higgs field $\phi$ on $V$, the Higgs bundle $(V,\phi)$ (resp., any of its pull-backs $f^\ast(V,\phi)$ along some morphism $f:Y\to X$) is semistable if and only if the vector bundle $V$ (resp., its pull-back $f^\ast V$) is.

Our results on Higgs bundles on elliptic surfaces allow us to prove (Proposition \ref{conj rel}) that Conjecture \ref{conjec} is true for Higgs bundles $(V,\phi)$ on $X$ {when  $V$ is} fiberwise semistable and {has} vertical determinant, as soon as the the rank $r$ and second Chern number $e$ of $V$ satisfy the inequality $e\geqslant rd+2g$, and the spectral curve of $V$ is sufficiently general.

The content of the paper is as follows: in Section \ref{uno} we make precise the class of elliptic surfaces we shall work with, and we prove a result (Proposition \ref{le potenze}) concerning the direct images of the (co)tangent sheaf of an elliptic surface belonging to this class (which will be interpreted in Section \ref{3} in terms of Hitchin bases); in Section \ref{due} we recall, following \cite{mfw}, the construction of the spectral cover of a fiberwise semistable vector bundle with vertical determinant on a Weierstrass fibration;
in Section \ref{3} we prove our results on the structure of Higgs fields on fiberwise semistable (resp., regular) bundles with vertical determinant on elliptic surfaces; finally, in Section \ref{ultima sezione}, we apply the results of Section \ref{3} to the study of Conjecture \ref{conjec}.

	  \subsection*{Acknowledgments} The second author wishes to thank the Department of Mathematics of Universidade Federal da Para\'iba, Jo\~ao Pessoa,  for hospitality; and Valeriano Lanza for the invitation to give a talk on a preliminary version of {this} paper at Universidade Federal Fluminense, Rio de Janeiro.

	  \subsection*{Notations and conventions}
We work over an algebraically closed field $k$ of characteristic $0$ (though at some point in Section \ref{ultima sezione} we specialize to the case $k=\mathbb{C}$). The word \emph{scheme} without further qualifiers will mean {\emph{separated $k$-scheme of finite-type}}.  Coherent sheaves of modules on a scheme $S$ will be referred to simply as sheaves; the dimensions (as $k$-vector spaces) of the cohomology groups (with respect to the Zariski topology) $ H^i(S,F)$ of a sheaf $F$ on $S$ will be denoted by $ h^i(S,F)$.  For a sheaf $F$ on $S$, 
we sometimes denote $F\otimes k(x)$ as $F(x)$, if $x$ is a point of $S$ (and $k(x)$ its residue field), and $F\otimes\mathcal{O}_S(D)$ as $F(D)$, if $D$ is a Cartier divisor on $S$. For a global section $s$ of $F$, $\mathrm{Z}(s)$ is the scheme of zeros of $s$.
Sheaves of K\"ahler differentials are denoted by $\Omega $, and their duals by $\Theta $; for a morphism of schemes $f:T\to S$, the morphism of structure sheaves $\mathcal{O}_S\to f_\ast\mathcal{O}_T$, and the canonical map $f^\ast\Omega_S\to\Omega_T$ will both be denoted by $f^\ast$; if $x\in S$ is any point, the scheme-theoretic fiber of $f$ {at} $x$ will be denoted by $T_x$, and the pull-back of a sheaf $F$ {to} $T$ via the canonical morphism $T_x\to T$ by $F_x$.
 For a locally free sheaf $V$ on a scheme $S$, we set $\mathbb{P}_S(V):=\mathrm{Proj}_S(\mathrm{Sym}V^\vee)$.
 
 \medskip
\section{Elliptic surfaces}
\label{uno}

\subsection{Elliptic fibrations}

Let $B$ be an integral scheme, with generic point $\eta$ and function field $K$. An \emph{elliptic fibration over $B$} is a proper and flat $B$-scheme $X\to B$, whose geometric generic fiber $X_\eta\times_KK^{\mathrm{al}}$ is an irreducible, complete, non-singular $K^{\mathrm{al}}$-curve of genus one, for some (hence, for any) algebraic closure $K^{\mathrm{al}}$ of $K$.
\emph{We assume, from now on, that the base and the total space of our elliptic fibrations are \emph{non-singular}}.

\subsubsection{Sections}

Let $\pi:X\to B$ be an elliptic fibration. A \emph{section} of $\pi$ is a closed subscheme  $\Sigma$ of $X$ such that the restriction of $\pi$ to $\Sigma$ is an isomorphism $\Sigma\to B$, or, equivalently, a morphism $\sigma:B\to X$ satisfying $\pi\circ\sigma=\mathrm{id}_B$.

Let $\Sigma$ be a section of $\pi$. Then $\Sigma$ intersects each closed fiber $X_b$ of $\pi$ transversely in the smooth point $\sigma(b)$; moreover, $\Sigma$ is a (non-singular, prime) divisor of $X$, and thus its conormal sheaf is invertible (and canonically isomorphic to $\mathcal{O}_\Sigma(-\Sigma):=\left.\mathcal{O}_X(-\Sigma)\right|_\Sigma$). The image in $\mathrm{Pic}(B)$ of $\mathcal{O}_\Sigma(-\Sigma)$ via the isomorphism $\pi|_\Sigma:\Sigma\to B$ (or, equivalently, the conormal sheaf of the closed immersion $\sigma:B\to X$) will be denoted by $\mathbb{L}$, and called the \emph{fundamental line bundle} of the fibration.  It turns out that $\mathbb{L}$ is isomorphic to $(\mathrm{R}^1\pi_\ast\mathcal{O}_X)^{\vee}$; in particular, it does not depend on the chosen section. We recall the following well-known {fact} \cite{miranda, libro friedman}:

\begin{Prop}
\label{immagini multipli sezione}
Let $\pi:X\to B$ be an elliptic fibration with section $\Sigma$. Then:
\begin{enumerate}
\item[(i)]  the map $\pi^\ast:\mathcal{O}_B\to\pi_\ast\mathcal{O}_X$ is an isomorphism;
\item[(ii)] for each non-zero integer $r$, the sheaf $\pi_\ast\mathcal{O}_X(r\Sigma)$ on $B$ is $0$ for $r<0$, while it is (non-canonically) isomorphic to
$$
\mathbb{L}^{-2}\oplus\cdots\oplus\mathbb{L}^{-r}\oplus\mathcal{O}_B
$$
for $r>0$.
\end{enumerate}

\end{Prop}

\subsection{Weierstrass fibrations}

A \emph{Weierstrass fibration} is a pair $(\pi,\Sigma)$, where $\pi$ is an elliptic fibration \emph{with integral fibers}, and $\Sigma$ is a section of $\pi$. 

Let $\mathbb{L}$ be the fundamental line bundle of a Weierstrass fibration $(\pi:X\to B,\Sigma)$. By the second item of Proposition \ref{immagini multipli sezione}, the sheaf $\pi_\ast\mathcal{O}_X(3\Sigma)$ is isomorphic to 
$$
\mathbb{L}^{-2}\oplus\mathbb{L}^{-3}\oplus\mathcal{O}_B=:\mathcal{S}_3.
$$
Set $\mathbb{P}:=\mathbb{P}_B(\mathcal{S}_3^\vee)$, and denote by $p:\mathbb{P}\to B$ the projection. Next, observe that
$$
 H^0(\mathbb{P},\mathcal{O}_\mathbb{P}(3)\otimes p^\ast\mathbb{L}^6)= H^0(B,\mathbb{L}^6\otimes\mathrm{Sym}^3\mathcal{S}_3)\simeq\bigoplus_{i+j+k=3} H^0(B,\mathbb{L}^{6-2i-3j})x^iy^jz^k.
$$
The last vector space contains
$$
 H^0(\mathcal{O}_B)y^2z\oplus H^0(\mathcal{O}_B)x^3\oplus H^0(\mathbb{L}^4)xz^2\oplus H^0(\mathbb{L}^6)z^3.
$$
It follows that, for each pair $(a_4,a_6)\in H^0(\mathbb{L}^4)\oplus H^0(\mathbb{L}^6)$, the expression
\begin{equation}
\label{wei equa}
y^2z-(x^3+a_4xz^2+a_6z^3)
\end{equation}
can be interpreted as a non-zero section of the line bundle $\mathcal{O}_\mathbb{P}(3)\otimes p^\ast\mathbb{L}^6$ on $\mathbb{P}$; thus its scheme of zeroes is an effective divisor $\bar X $ on $\mathbb{P}$. Moreover, if the \emph{discriminant} section
\begin{equation}
\label{il discri}
\Delta:=-16(4a_4^3+27a_6^2)\in H^0(\mathbb{L}^{12})
\end{equation}
is non-zero, then the pair $(\bar{\pi},\bar{\sigma})$, where $\bar{\pi}:\bar X \to B$ is the composition
\begin{equation}
\label{can w fib}
\bar X \hookrightarrow\mathbb{P}\xrightarrow{p}B,
\end{equation}
while $\bar{\sigma}:B\to\bar X $ is obtained by factoring, via $\bar X \hookrightarrow \mathbb{P}$, the section of $p$ corresponding to the natural map $\mathcal{S}_3\twoheadrightarrow\mathbb{L}^{-3}$, is a Weierstrass fibration on $B$. The divisor $\bar X $ 
may be singular, but a local computation shows that it is actually non-singular provided the divisor of the discriminant section is reduced. Conversely, one has the following result \cite{miranda, libro friedman}:

\begin{Prop}
\label{sull w fib}
Let $(\pi:X\to B,\Sigma)$ be a Weierstrass fibration, and $\mathbb{L}$ its fundamental line bundle. Then the divisor $3\Sigma$ on $X$ is very ample relatively to $\pi$, and there exists an isomorphism $\pi_\ast\mathcal{O}_X(3\Sigma)\simeq\mathbb{L}^{-2}\oplus\mathbb{L}^{-3}\oplus\mathcal{O}_B=:\mathcal{S}_3$, such that the image of the immersion of $B$-schemes $X\to\mathbb{P}_B(\mathcal{S}_3^\vee)$ corresponding to the surjection
$$
\pi^\ast \mathcal{S}_3\xrightarrow{\simeq}\pi^\ast\pi_\ast\mathcal{O}_X(3\Sigma)\xrightarrow{\mathrm{ev}}\mathcal{O}_X(3\Sigma)
$$
is the divisor of zeros of \eqref{wei equa} for suitable $a_i\in H^0(\mathbb{L}^i)$, $i=4,6$. Moreover, the pair $(a_4,a_6)$ is well defined up to the $k^\times $-action $c\cdot(a_4,a_6):=(c^4a_4,c^6a_6)$.

\end{Prop}

In the {remainder} of {this} paper we shall denote by $B$ an irreducible, smooth, projective \emph{curve} of genus $g$, and by $\pi:X\to B$ a Weierstrass fibration, with section $\Sigma$, and fundamental line bundle $\mathbb{L}$. Moreover, we set
$$
d:=\mathrm{deg}(\mathbb{L}),
$$
and we remark that, as a consequence of Proposition \ref{sull w fib}, we have $d\geqslant 0$, and $d=0$ if and only if  $\mathbb{L}^{12}\simeq\mathcal{O}_B$. \emph{We will always assume that $\pi$ has at worst nodal fibers} (for example, for a Weierstrass fibration as in \eqref{can w fib}, this is equivalent to the divisors of the sections $a_4,a_6$ having disjoint supports).

Let $Z$ be the scheme of singularities of $\pi$. Then $Z$ is a local complete intersection closed subscheme of $X$ of dimension $0$ and length $12d$, supported on the nodes of the singular fibers, and $\pi$ maps $Z$ isomorphically onto the scheme of zeros of the discriminant section $\Delta$ from \eqref{il discri}. The us state a lemma which will be useful in the following.

\begin{lemma}
\label{le prop del fascio di id di Z}
For each $n\geqslant 0$
the sheaf $\mathrm{Sym}^n\mathcal{I}_Z$ is torsion-free.
\end{lemma}

\begin{proof}
It is enough to prove the following algebraic statement: \emph{let $R$ be a UFD, and let $I$ be the ideal of $R$ generated by two relatively prime elements $f,g\in R$. Then the symmetric powers $\mathrm{Sym}^r I$ of $I$ are torsion-free $R$-modules}. This is trivial for $r\in\{0,1\}$. So, let us fix an $r\geqslant 2$, and set $M:=\mathrm{Sym}^rI$. The claim is then equivalent to the map $\mathrm{ev}_M:M\to M^{\vee\vee}$ being injective. From the presentation (actually, free resolution) of $I$
$$
F_1\xrightarrow{\beta}F_0\xrightarrow{\alpha} I\to 0,
$$
where $F_0,F_1$ are free $R$-modules with bases $(e_1,e_2)$ and $(e)$ respectively, while $\alpha$ and $\beta$ are the maps
$$
\alpha:\left\{
\begin{array}{c}
e_1\mapsto f\\
e_2\mapsto g
\end{array}\right.,
\qquad\beta:e\mapsto ge_1-fe_2,
$$
we get the presentation of $M$
\begin{equation}
\label{presentazione di M}
G_1\xrightarrow{\delta} G_0\xrightarrow{\gamma} M\to 0,
\end{equation}
where the modules $G_0:=\mathrm{Sym}^rF_0$ and $ G_1:=F_1\otimes\mathrm{Sym}^{r-1}F_0$ are free with bases 
$$
(u_i:=e_1^ie_2^{r-i})_{0\leqslant i\leqslant r},\qquad (v_i:=e\otimes e_1^ie_2^{r-1-i})_{0\leqslant i\leqslant r-1}
$$
respectively, $\gamma:=\mathrm{Sym}^r\alpha$, and $\delta$ is induced by the multilinear map (symmetric in the $y_i$'s)
$$
F_1\times F_0^{r-1}\to G_0,\qquad (x,y_1,\dots,y_{r-1})\mapsto \beta(x)y_1\cdots y_{r-1};
$$
thus, with respect to the bases $(u_i), (v_i)$, $\delta$ is represented by the $(r+1)\times r$ matrix with entries in $R$
$$
\left(
\begin{array}{ccccc}
-f & &  & &\\
g & -f &  & \\
& g & \ddots & \\
&& \ddots &-f & \\
 && &g & -f\\
 & &&& g
\end{array}
\right).
$$
Dualizing \eqref{presentazione di M}, we obtain the exact sequence
$$
0\to M^\vee\xrightarrow{\gamma^\vee}G_0^\vee\xrightarrow{\delta^\vee}G_1^\vee,
$$
showing that $\gamma^\vee$ maps $M^\vee$ isomorphically onto the kernel of $\delta^\vee$. To compute this kernel,  let $K$ be the field of fractions of $R$. An easy computation shows that $\ker(\delta^\vee\otimes_RK)$ is generated by $\omega\otimes 1\in G_0^\vee\otimes_RK$, where
$$
\omega:=\sum_{i=0}^rf^{i}g^{r-i}u_i^\vee\in G_0^\vee;
$$
a standard argument (based on the assumptions that $R$ is factorial, and that $\mathrm{gcd}(f,g)=1$) then shows that $\ker(\delta^\vee)=R\omega$. Thus $M^\vee =R\bar{\omega}$, where $\bar{\omega}\in M^\vee$ is the preimage of $\omega$ via $\gamma^\vee$. This implies, in particular, that $\ker(\mathrm{ev}_M)=\ker(\bar{\omega})=\gamma(\ker(\omega))$. Now,  $\ker(\omega)$ consists of the elements $x:=\sum_{i=0}^ra_iu_i$ of $G_0$ ($a_i\in R$) such that the pair $(f,g)$ is a root of the (homogeneous, degree $r$) polynomial $
\phi:=\sum_{i=0}^ra_iX^iY^{r-i}\in R[X,Y]$. An easy argument (again based on $R$ being a UFD) shows that $\phi(f,g)=0$ if and only if $\phi$ factors in $R[X,Y]$  as $(gX-fY)\cdot\psi$, for some homogeneous, degree $r-1$ polynomial $\psi$, say $\psi=\sum_{i=0}^{r-1}b_iX^iY^{r-1-i}$. Expanding the equation $\phi=(gX-fY)\cdot\psi$ then shows that the element $y:=\sum_{i=0}^{r-1}b_iv_i\in G_1$ satisfies $\delta(y)=x$. Thus $\ker(\omega)\subseteq\mathrm{im}(\delta)$, and $\mathrm{ev}_M$ is injective, as claimed.
\end{proof}

\subsection{Differentials}

Let $\Omega_\pi$ (resp., $\omega_\pi^\circ$) be the sheaf of relative K\"ahler differentials (resp., the relative dualizing sheaf) of the projection $\pi$. The following useful isomorphisms will be used without further reference in the sequel:
$$
\Omega_\pi\simeq\mathcal{I}_Z\otimes\omega_\pi^\circ,
\qquad
\omega_\pi^\circ\simeq\pi^\ast\mathbb{L}\in\mathrm{Pic}(X).
$$
They show, in particular, that $\Omega_\pi$ is torsion-free,
with determinant $\det\Omega_\pi\simeq\omega_\pi^\circ$,
and that, for each closed fiber $X_b$ of $\pi$, the restriction $\mathcal{I}_{Z,b}$ of $\mathcal{I}_Z$ to $X_b$ is isomorphic to the sheaf of K\"ahler differentials of $X_b$; thus, $ h^1(X_b,\mathcal{I}_{Z,b})= h^1(X_b,\Omega_{X_b})=1$. The last equality is obvious for $X_b$ smooth, while, for $X_b$ singular, it can be seen, for example, as a consequence of the following {fact}.
\begin{lemma}
Let $C$ be an integral, projective, nodal curve of arithmetic genus $1$, and let $x\in C(k)$ be the node of $C$. Then there is an exact sequence of sheaves on $C$
\begin{equation}
\label{il cotangentente della nodale}
0\to k(x)\to\Omega_C\to\mathfrak{m}_x\to 0.
\end{equation}
\end{lemma}

{The sheaf $\mathrm{R}^1\pi_\ast\mathcal{I}_Z$ on $B$ is  thus invertible by
Grauert's theorem, and it  can be computed explicitly using relative Serre duality \cite{kleiman}}:
\begin{eqnarray*}
\mathrm{R}^1\pi_\ast\mathcal{I}_Z
&\simeq &
\mathit{Ext}_\pi^0(\mathcal{I}_Z,\omega_\pi^\circ)^\vee\simeq
(\pi_\ast\mathit{Hom}_X(\mathcal{I}_X,\pi^\ast\mathbb{L}))^\vee\\
&\simeq & \pi_\ast(\mathcal{I}_Z^\vee\otimes\pi^\ast\mathbb{L})^\vee\simeq \mathbb{L}^{-1}\otimes\pi_\ast\mathcal{O}_X\simeq\mathbb{L}^{-1}.
\end{eqnarray*}
The last isomorphism allows us to prove the following {result.}

\begin{lemma}
\label{immagine diretta del fascio di ideali dello schema singolare}
There is a canonical isomorphism 
$$
\pi_\ast\mathcal{I}_Z\simeq\mathbb{L}^{-12}.
$$
\end{lemma}

\begin{proof}

We have the exact sequence of sheaves on $X$:
$$
0\to\mathcal{I}_Z\to\mathcal{O}_X\to\mathcal{O}_Z\to 0;
$$
thus, setting $Z':=\mathrm{Z}(\Delta)$, and taking the exact sequence of higher direct $\pi$-images,
we get an exact sequence
$$
0\to\pi_\ast\mathcal{I}_Z\to\mathcal{O}_B\to\mathcal{O}_{Z'}\to\mathbb{L}^{-1}\to\mathbb{L}^{-1}\to 0,
$$
showing that the map $\mathcal{O}_B\to\mathcal{O}_{Z'}$ is surjective. But, up to non-zero scalars, there is only one surjection $\mathcal{O}_B\to\mathcal{O}_{Z'}$; so
$$
\pi_\ast\mathcal{I}_Z\simeq\mathcal{O}_B(-Z')\simeq\mathbb{L}^{-12},
$$
as claimed.
\end{proof}

\subsubsection{} In the canonical exact sequence of sheaves of K\"ahler differentials
$$
\pi^\ast\Omega_B\xrightarrow{\pi^\ast} \Omega_X\to\Omega_\pi\to 0,
$$
the morphism $\pi^\ast$ is generically injective by rank counting, so 
that (the sheaf $\pi^\ast\Omega_B$ being locally free on the integral scheme $X$) the sequence
\begin{equation}
\label{exact sequence of kahler differntials}
\mathcal{E}:\quad 0\to\pi^\ast\Omega_B\xrightarrow{\pi^\ast} \Omega_X\to\Omega_\pi\to 0
\end{equation}
is exact. In particular, one has the following special case of \emph{Kodaira's formula} \cite{miranda} for the canonical line bundle of $X$
$$
\omega_X=\det\Omega_X\simeq\pi^\ast\omega_B\otimes\det\Omega_\pi\simeq\pi^\ast(\omega_B\otimes\mathbb{L}).
$$

\subsubsection{}
 For a closed point $b$ of $B$, the restriction
$$
\mathcal{E}_b:\quad 0\to\mathcal{O}_{X_b}\otimes_{k(b)}\Omega_B(b)\to\Omega_{X,b}\to\Omega_{X_b}\to 0
$$
of $\mathcal{E}$ to $X_b$ is exact (by the usual arguments); one can thus consider the \emph{relative extension class} of \eqref{exact sequence of kahler differntials}, or \emph{Kodaira-Spencer map} of $\pi$ \cite{huy}. This is a global section $\xi(\mathcal{E}/B)$ of the sheaf on $B$
\begin{eqnarray*}
\mathit{Ext}^1_\pi(\Omega_\pi,\pi^\ast\Omega_B)&\simeq &\mathit{Ext}_\pi^1(\Omega_\pi\otimes\pi^\ast\omega_B^{-1}\otimes\omega_\pi^\circ,\omega_\pi^\circ)\simeq \mathit{Ext}_\pi^1(\pi^\ast(\mathbb{L}^{2}\otimes\omega_B^{-1})\otimes\mathcal{I}_Z,\omega_\pi^\circ)\\
&\simeq &(\mathrm{R}^0\pi_\ast(\pi^\ast(\mathbb{L}^{2}\otimes\omega_B^{-1})\otimes\mathcal{I}_Z))^\vee
\simeq
\mathbb{L}^{-2}\otimes\omega_B\otimes(\pi_\ast\mathcal{I}_Z)^\vee\simeq\mathbb{L}^{10}\otimes\omega_B
\end{eqnarray*}
(the last isomorphism is due to Lemma \ref{immagine diretta del fascio di ideali dello schema singolare}),
having the property that, for each $b\in B(k)$, the extension $\mathcal{E}_b$ splits if and only if the image of $\xi(\mathcal{E}/B)$
in $\mathbb{L}^{10}\otimes\omega_B\otimes k(b)$ is zero. It is clear that $\mathcal{E}_b$ is non-split if $X_b$ is singular. It follows that if our fibration has at least one singular fiber (which is equivalent to $d>0$; one says in this case that the fibration is \emph{non-isotrivial}), then $\xi(\mathcal{E}/B)$ is a non-zero section of the line bundle $\mathbb{L}^{10}\otimes\omega_B$, and thus, for almost all $b\in B(k)$, the extension $\mathcal{E}_b$ is non-split. For such a $b$, the bundle $\Omega_{X,b}$ is a then a non-split self-extension of $\mathcal{O}_{X_b}$. Weak isomorphism classes of such extensions correspond bijectively to closed points of the projective space $\mathbb{P}\mathrm{Ext}^1_{X_b}(\mathcal{O}_{X_b},\mathcal{O}_{X_b})\simeq\mathbb{P} H^1(\mathcal{O}_{X_b})$, which is zero-dimensional; the sheaf corresponding to its closed point is the rank $2$ \emph{Atiyah bundle}  $\mathrm{I}_2$  of $X_b$ \cite{atiyah}. \emph{From now on we will assume our fibration to be non-isotrivial, so that $\Omega_{X,b}$ will be isomorphic to the rank $2$ Atiyah bundle on $X_b$, for $b\in B(k)$} general.

\subsubsection{}

We remark that, for each $r\geqslant 0$, we have a natural map
\begin{equation}
\label{la mappa da sym bla a sym bla}
\pi^\ast\omega_B^r\xrightarrow{\simeq}\pi^\ast\mathrm{Sym}^r\Omega_B\xrightarrow{\simeq}\mathrm{Sym}^r\pi^\ast\Omega_B\xrightarrow{\mathrm{Sym}^r\pi^\ast}\mathrm{Sym}^r\Omega_X.
\end{equation}
We then have the following Proposition, which will be applied in Section \ref{3} to the study of Higgs bundles on elliptic surfaces:

\begin{Prop}
\label{le potenze}
For each $r\geqslant 0$,  the adjoint
$$
\omega_B^r\to\pi_\ast\mathrm{Sym}^r\Omega_X
$$
of the map \eqref{la mappa da sym bla a sym bla} is an isomorphism; furthermore, there is a canonical isomorphism
$$
\pi_\ast\mathrm{Sym}^r\Theta_X\simeq\mathbb{L}^{-r}.
$$
\end{Prop}

Its proof we will make use of the following elementary lemma, which will also be applied multiple times in the sequel:

\begin{lemma}
\label{lemma utile}
Let $F$ be a torsion-free sheaf on $X$. Then its direct image $\pi_\ast F$ is a locally free sheaf on $B$.
\end{lemma}

\begin{proof}[Proof of Proposition \ref{le potenze}]
We will prove the claim by induction on $r$. The case $r=0$ is just the statement that the map $\pi^\ast:\mathcal{O}_B\to\pi_\ast\mathcal{O}_X$ is an isomorphism, which we already discussed (see Proposition \ref{immagini multipli sezione}). So, let us fix an integer $r\geqslant 1$, and let us assume the claim to be true for $r-1$. Starting from \eqref{exact sequence of kahler differntials}, we get a canonical exact sequence of sheaves on $X$:
\begin{equation}
\label{successione che mi serve}
\mathrm{Sym}^{r-1}\Omega_X\otimes\pi^\ast\omega_B\to\mathrm{Sym}^r\Omega_X\to\mathrm{Sym}^r\Omega_\pi\to 0,
\end{equation}
in which the first two sheaves are locally free of rank $r$ and $r+1$, respectively, while the third
$$
\mathrm{Sym}^r\Omega_\pi\simeq\mathrm{Sym}^r(\mathcal{I}_Z\otimes\pi^\ast\mathbb{L})\simeq\mathrm{Sym}^r\mathcal{I}_Z\otimes\pi^\ast\mathbb{L}^r
$$
it torsion-free of rank $1$ by Lemma \ref{le prop del fascio di id di Z}. The usual arguments then show that the first map in \eqref{successione che mi serve} is injective. We thus have a short exact sequence of torsion-free sheaves on $X$:
$$
0\to \mathrm{Sym}^{r-1}\Omega_X\otimes\pi^\ast\omega_B\to\mathrm{Sym}^r\Omega_X\to\mathrm{Sym}^r\Omega_\pi\to 0;
$$
taking direct images we get an exact sequence of (locally free, by Lemma \ref{lemma utile}) sheaves on $B$
$$
0\to\omega_B\otimes\pi_\ast\mathrm{Sym}^{r-1}\Omega_X\to\pi_\ast\mathrm{Sym}^{r}\Omega_X\to\mathbb{L}^r\otimes\pi_\ast\mathrm{Sym}^r\mathcal{I}_Z.
$$
All the bundles appearing {here} actually  have  rank $1$ (this is clear for $\mathbb{L}^r\otimes\pi_\ast\mathrm{Sym}^r\mathcal{I}_Z$. For $\omega_B\otimes\pi_\ast\mathrm{Sym}^{r-1}\Omega_X$ and $\pi_\ast\mathrm{Sym}^r\Omega_X$ one uses that the restriction of $\Omega_X$ to $X_b$ is isomorphic to $\mathrm{I}_2$ for $b$ general; thus
$$
(\mathrm{Sym}^r\Omega_X)_b\simeq\mathrm{Sym}^r\mathrm{I}_2\simeq\mathrm{I}_{r+1},
$$
the Atiyah bundle of rank $r+1$ on $X_b$, which satisfies $ h^0(X_b,\mathrm{I}_{r+1})=1$). An easy argument now shows that the map $\omega_B\otimes\pi_\ast\mathrm{Sym}^{r-1}\Omega_X\to\pi_\ast\mathrm{Sym}^{r}\Omega_X$ is an isomorphism. The inductive hypothesis guarantees that the first arrow in 
$$
\omega_B^r=\omega_B\otimes\omega_B^{r-1}\to\omega_B\otimes\pi_\ast\mathrm{Sym}^{r-1}\Omega_X\to\pi_\ast\mathrm{Sym}^{r}\Omega_X,
$$
is an isomorphism. This completes the proof of the first claim, since this composition is nothing but the adjoint of \eqref{la mappa da sym bla a sym bla}.

The second claim is proved by taking the direct image of the isomorphism
$$
\mathrm{Sym}^r\Theta_X\simeq\mathrm{Sym}^r(\Omega_X\otimes\omega_X^{-1})\simeq\mathrm{Sym}^r\Omega_X\otimes\pi^\ast(\mathbb{L}\otimes\omega_B)^{-r},
$$
and then using the projection formula and the first claim.
\end{proof}

\medskip
\section{Vector bundles on elliptic surfaces}
\label{due}

Let $(\pi:X\to B,\Sigma)$ be a non-isotrivial Weierstrass fibration with nodal singular fibers. We will be interested in vector bundles $V$ on $X$ whose restrictions to the closed fibers of $\pi$ have trivial determinant, and are semistable (we recall that there is, essentially, only one notion of semistability for a torsion-free sheaf on an integral, complete curve $C$, which does not require the choice of a polarization on $C$; also, if $C$ has arithmetic genus $1$, the \emph{degree} of a sheaf on $C$ is just its Euler-Poincar\'e characteristic).

Following \cite{mfw}, we will recall how it is possible to associate, to each such $V$, a finite map $C_V\to B$ of degree equal to the rank of $V$, called the \emph{spectral  cover of $V$}. Properties of $C_V$ such as, for instance, reducedness and integrality, strongly influence those of $V$; in particular, we will see in the next section that they restrict the types of Higgs fields which $V$ can support.

\subsection{Spectral covers}

Let $V$ be a rank $r$ bundle on $X$, and let us assume that the restriction of $V$ to each closed fiber $X_b$ of $\pi$ is semistable and has trivial determinant. Then, first of all, the sheaf $\delta:=\pi_\ast\det V$ on $B$ is invertible, and the natural map
$$
\pi^\ast\delta=\pi^\ast\pi_\ast\det V\to\det V
$$
is an isomorphism (one also says that $\det V$ is \emph{vertical} in this case). Moreover, the twisted bundle
$$
V(\Sigma):=V\otimes\mathcal{O}_X(\Sigma)
$$
restricts to a semistable, degree $r$ bundle on each $X_b$, and thus satisfies $
 h^0(X_b,V(\Sigma)_b)=r$, $ h^1(X_b,V(\Sigma)_b)=0$. It follows that the sheaf $\mathrm{R}^0\pi_\ast V(\Sigma)$ on $B$ is locally free of rank $r$, while the sheaf $\mathrm{R}^1\pi_\ast V(\Sigma)$ is $0$. Let us recall the following result from \cite{mfw}:

\begin{Prop}
For each closed point $b$ of $B$, the restriction of the natural map
\begin{equation}
\label{valua}
\mathrm{ev}:\pi^\ast\pi_\ast V(\Sigma)\to V(\Sigma)
\end{equation}
to the fiber $X_b$ of $\pi$ over $b$ is an isomorphism at the generic point of $X_b$.
\end{Prop}

This implies, in particular, that the determinant of \eqref{valua} is a non-zero map of line bundles on $X$, from
$$
\det\pi^\ast\pi_\ast V(\Sigma)\simeq\pi^\ast\det\pi_\ast V(\Sigma)
$$
to
$$
\det V(\Sigma)\simeq\det V\otimes\mathcal{O}_X(r\Sigma)\simeq\pi^\ast\delta\otimes\mathcal{O}_X(r\Sigma);
$$
thus, setting
\begin{equation}
\label{def of mu}
\mu:=(\det\pi_\ast V(\Sigma))^{-1}\otimes\delta=(\det\pi_\ast V(\Sigma))^{-1}\otimes\pi_\ast\det V\in\mathrm{Pic}(B),
\end{equation}
$\mathrm{det}(\mathrm{ev})$ can be considered as a non-zero section of the line bundle $\mathcal{O}_X(r\Sigma)\otimes\pi^\ast\mu$ on $X$, well defined up to multiplication by elements of $k^\times$; or as a curve $C_V$ (which is, in general, neither reduced nor irreducible) belonging to the linear system on $X$
$$
\left|r\Sigma+\pi^\ast\mu\right|.
$$
$C_V$ is called the \emph{spectral curve} of $V$;
the restriction of $\pi$ to $C_V$ is a finite morphism of degree $r$ from $C_V$ to $B$, called the \emph{spectral cover} of $V$. We remark that a slight modification of the previous construction allows one to define a spectral cover also in the case in which $V$ is assumed to be semistable not on every fiber, but only on a general fiber.

\subsubsection{}
 Set $\mathcal{S}_r:=\mathrm{R}^0\pi_\ast\mathcal{O}_X(r\Sigma)\simeq \bigoplus_{i\in\{2,\dots,r,0\}} \mathbb{L}^{-i}$. Then, by virtue of the isomorphism
$$
 H^0(X,\mathcal{O}_X(r\Sigma)\otimes\pi^\ast\mu)\simeq H^0(B,\mathcal{S}_r\otimes\mu),
$$
the map $\mathrm{\det}(\mathrm{ev})$ can also be considered as a
section of the bundle
$$
\mu\otimes \mathcal{S}_r\simeq\bigoplus_{i\in\{2,\dots,r,0\}}\mu\otimes \mathbb{L}^{-i};
$$
or, equivalently, as an $r$-tuple
$$
(s_2,\dots,s_r,s_0)\in\bigoplus_{i\in\{2,\dots,r,0\}} H^0(B,\mu\otimes\mathbb{L}^{-i}),
$$
well defined up to an overall scaling by a non-zero element of $k$.
It follows that the vector space $ H^0(B,\mu\otimes\mathbb{L}^{-i})$ is non-zero for some $i\in\{2,\dots,r,0\}$; thus, setting
$$
e:=\deg(\mu),
$$
we have $0\leqslant \deg(\mu\otimes\mathbb{L}^{-i})=e-id$, or $e\geqslant id$. In particular, $e$ is non-negative. The following lemma shows that this is nothing but a Bogomolov-type inequality (see also Section \ref{ultima sezione}, and \cite{mfw}, where the same result is proved with different techniques under the assumption of fiberwise \emph{regularity} on almost all fibers; we will go back later in this section to the notion of regularity):

\begin{lemma}
\label{e uguale c2}
Let $V$ be a locally free sheaf on $X$, fiberwise semistable and with vertical determinant. Then the degree $e$ of the line bundle $\mu$ defined in \eqref{def of mu} satisfies
$$
e=\mathrm{c}_2(V).
$$
\end{lemma}

\begin{proof}
By \eqref{def of mu}, we have
\begin{equation}
\label{inter 2}
e=\deg(\delta)-\deg(\det\pi_\ast V(\Sigma))=\deg(\delta)-\deg(\pi_\ast V(\Sigma)),
\end{equation}
The degree of $\pi_\ast V(\Sigma)$ can be computed, for example, using Grothendieck-Riemann-Roch theorem: 
$$
\mathrm{ch}(\pi_!V(\Sigma))\mathrm{td}(B)=\pi_\ast(\mathrm{ch}(V(\Sigma))\mathrm{td}(X)).
$$
One finds that $\deg(\pi_\ast V(\Sigma))=\deg(\delta)-\mathrm{c}_2(V)$, which, substituted in \eqref{inter 2}, gives the claim $e=\mathrm{c}_2(V)$.

\end{proof}

Let us remark that if the second Chern number of $V$ satisfies the inequality 
$$
\mathrm{c}_2(V)=e<id
$$
for some $0< i\leqslant r$, then
$$
\deg(\mu\otimes\mathbb{L}^{-j})=e-jd\leqslant e-id<0
$$
for all indexes $j\geqslant i$. Thus, the component $s_j\in H^0(\mu\otimes\mathbb{L}^{-j})$ of $\mathrm{det}(\mathrm{ev})=(s_2,\dots,s_r,s_0)$ vanishes for $j\geqslant i$, and so the spectral curve of $V$ belongs to the image of the canonical injection of linear series
$$
\left|(i-1)\Sigma+\pi^\ast\mu\right|\hookrightarrow \left|r\Sigma+\pi^\ast\mu\right|,\qquad D\mapsto D+(r-i+1)\Sigma.
$$
In particular, for $i=r$, we find that $C_V\geqslant \Sigma$ is non-integral, while for $i=r-1$, $C_V\geqslant 2\Sigma$ is non-reduced. {So we have:}
\begin{lemma}
\label{le prop di spec cover}
The following implications hold:
\begin{enumerate}
\item[(i)] if $C_V$ is integral, then $\mathrm{c}_2(V)\geqslant rd$;
\item[(ii)] if $C_V$ is reduced, then $\mathrm{c}_2(V)\geqslant (r-1)d$.
\end{enumerate}

\end{lemma}

\subsection{Fiberwise regular bundles}

In this section we recall two properties, both from \cite{mfw}, of fiberwise regular bundles with vertical determinant on $X$, which will be useful in the following sections. 

First of all, let us recall the definition of \emph{regular bundle} on a fiber of $\pi$. To this end, let us fix an   integral, complete curve $E$ of arithmetic genus $1$ (thus $E$ is either non-singular, or it has exactly one node, or its has exactly one cusp). The degree of a coherent sheaf on $E$ is, by definition, its Euler-Poincar\'e characteristic. Thus one has the notions of slope, and hence of slope-(semi)stability, for coherent sheaves on $E$. 

Now let $\bar{J}$ be the set of isomorphism classes of torsion-free sheaves of rank $1$ and degree $0$ on $E$. If we denote by $J$   the group of isomorphism classes of line bundles of degree $0$ on $E$, then $J=\bar{J}$ if $E$ is non-singular, while in the nodal and cuspidal cases one has $\bar{J}=J\cup\{\mathcal{G}\}$ for some sheaf $\mathcal{G}\not\in J$.

Next, let us fix an element $\lambda$ of $\bar{J}$. A vector bundle $\mathcal{I}$ on $E$ is said to be \emph{strongly indecomposable (with Jordan-H\"older constituent $\lambda$)} if it admits a filtration whose subquotients are all isomorphic to the sheaf $\lambda$, and if, in addition, the equality
 $$
 \dim(\mathrm{Hom}(\lambda,\mathcal{I}))=1
 $$
 holds. As the name suggests, strong indecomposability implies indecomposability (with respect to the operation of direct sum); furthermore, if the sheaf $\lambda$ is invertible, then any indecomposable bundle admitting a filtration with subquotients isomorphic to $\lambda$ turns out to be strongly indecomposable (and, in fact, it is isomorphic to the \emph{Atiyah bundle} $\mathrm{I}_r(\lambda)$ \cite{atiyah} of the appropriate rank).

Finally, a vector bundle of degree $0$ on $E$ is said to be \emph{regular} 
if it isomorphic to a direct sum of strongly indecomposable bundles having pairwise distinct Jordan-Holder constituents. 
 
We are now ready to state the two properties of fiberwise regular sheaves on elliptic surfaces which we mentioned at the beginning of the section. The first of these will be used in the proof of Proposition \ref{uno dei risultati}:

\begin{Prop}
\label{sui fib reg}
Let $V$ be a vector bundle of rank $r$ on $X$, which has vertical determinant and is fiberwise regular. Let $\varphi:C\to B$ be the spectral cover of $V$. Then there is an isomorphism of rank $r$ bundles on $B$ (and of sheaves of $\mathcal{O}_B$-algebras)
$$
\pi_\ast\,\mathit{End}(V)\simeq\varphi_\ast\mathcal{O}_C.
$$
\end{Prop}

%
Finally, the following sufficient condition of regularity on a fiber will play a role in the proof Proposition \ref{conj rel}:

\begin{Prop}
\label{sem implies reg}
Let $V$ be a fiberwise semistable bundle on $X$ with vertical determinant, and let $C\to B$ be the spectral cover of $V$. Furthermore, let $b$ be a closed point of the base $B$ such that the curve $C$ is non-singular at all points lying over $b$. Then the restriction of $V$ to the fiber $X_b$ is regular. In particular, if $C$ is non-singular, then the bundle $V$ is fiberwise regular.
\end{Prop}

\medskip
\section{Higgs bundles on elliptic surfaces}
\label{3}

\subsection{Higgs bundles}

Let $Y$ be a non-singular scheme, and let $V$ be a locally free sheaf on $Y$. A Higgs field on $V$ \cite{simp 1, simp 2, hitchin} is a morphism of $\mathcal{O}_Y$-modules $
\phi:V\to V\otimes\Omega_Y$, i.e., a global section of the bundle $\mathit{End}(V)\otimes\Omega_Y$, satisfying the \emph{integrability condition} (this is automatic if $Y$ is a curve)
$$
\phi\wedge\phi =0\quad\mbox{in}\quad\mathrm{Hom}(V,V\otimes\bigwedge^2\Omega_Y),
$$
where $\phi\wedge\phi:V\to V\otimes\bigwedge^2\Omega_Y $ denotes the composition
$$
V\xrightarrow{\phi}V\otimes\Omega_Y\xrightarrow{\phi\otimes 1} V\otimes\Omega_Y\otimes \Omega_Y\xrightarrow{1\otimes\wedge} V\otimes\bigwedge^2\Omega_Y;
$$
a Higgs bundle on $Y$ is a pair $(V,\phi)$, where $V$ is a locally free sheaf on $Y$, and $\phi$ a Higgs field on $V$. The dual of a map $\phi:V\to V\otimes\Omega_Y$ can also be regarded as a map $\Theta_Y\to\mathit{End}(V)$, 
which   induces a morphism of sheaves of (not necessarily commutative) $\mathcal{O}_Y$-algebras from the tensor algebra $\mathrm{T}\Theta_Y:=\bigoplus_{i\in\mathbb{Z}_{\geqslant 0}}\Theta_Y^{\otimes i}$ of $\Theta_Y$ to $\mathit{End}(V)$. The integrability condition is then equivalent to the requirement that the last morphism factors through the projection $\mathrm{T}\Theta_Y\to\mathrm{Sym}\,\Theta_Y$, thus defining a structure of $\mathrm{Sym}\,\Theta_Y$-module on $V$. 

{This} definition can be generalized, in a straightforward way, to that of \emph{Higgs sheaf}. 

\subsubsection{Operations on Higgs bundles}
\label{operations}
We now briefly recall some operations on Higgs bundles which will be useful in Section \ref{ultima sezione}. First of all, one has a natural notion of \emph{pull-back} for Higgs bundles: for a morphism $f:Z\to Y$ of non-singular schemes, and for a Higgs bundle $\mathcal{V}=(V,\phi)$ on $Y$, one sets $f^\ast\mathcal{V}:=(f^\ast V,\psi)$, where $\psi:f^\ast V\to f^\ast V\otimes\Omega_Z$ is the composition
$$
f^\ast V\xrightarrow{f^\ast\phi}f^\ast(V\otimes\Omega_Y)\xrightarrow{\simeq}f^\ast V\otimes f^\ast\Omega_Y\xrightarrow{1\otimes f^\ast}f^\ast V\otimes\Omega_Z.
$$
There is also a natural notion of \emph{tensor product} of two Higgs bundles defined on the same scheme, say $\mathcal{V}_1=(V_1,\phi_1),\mathcal{V}_2=(V_2,\phi_2)$: one sets
$$
\mathcal{V}_1\otimes\mathcal{V}_2:=(V_1\otimes V_2,\phi_1\otimes 1+1\otimes \phi_2).
$$
Finally, one can define the \emph{dual} of a Higgs bundle $\mathcal{V}=(V,\phi)$ on a scheme $Y$ as $\mathcal{V}^\vee:=(V^\vee,-\psi)$, where $\psi:V^\vee\to V^\vee\otimes\Omega_Y$ is the composition
$$
V^\vee\xrightarrow{\simeq}V^\vee\otimes\mathcal{O}_Y\xrightarrow{1\otimes\mathrm{tr}^\vee} V^\vee\otimes\Omega_Y^\vee\otimes\Omega_Y\xrightarrow{\phi^\vee\otimes 1}V^\vee\otimes\Omega_Y.
$$

\subsubsection{The cone of Higgs fields on $V$}
Let $E:=\mathit{End}(V)$. Then the map
$$
 H^0( E\otimes\Omega_Y)\to
 H^0( E\otimes\bigwedge^2\Omega_Y),\qquad\phi\mapsto\phi\wedge\phi,
$$
factors as
$$
 H^0( E\otimes\Omega_Y)\xrightarrow{\phi\mapsto\phi^2}\mathrm{Sym}^2 H^0( E\otimes\Omega_Y)\xrightarrow{\ell} H^0( E\otimes\bigwedge^2\Omega_Y),
$$
where $\ell$ is $k$-linear. The image of the dual map $\ell^\vee: H^0( E\otimes\bigwedge^2\Omega_Y)^\vee\to \mathrm{Sym}^2 H^0( E\otimes\Omega_Y)^\vee$ defines a linear system of quadrics in the projective space $\mathbb{P}( H^0( E\otimes\Omega_Y))$; let $\mathfrak{B}$ be its base scheme. Then the set $\mathcal{H}_V$ of Higgs fields on $V$ is just the set of closed points of the affine cone over the projective scheme $\mathfrak{B}$. The cone $\mathcal{H}_V$ might very well be a vector subspace of $ H^0( E\otimes\Omega_Y)$, and in fact this is what will happen in the cases we will analyze in the following.

\subsubsection{The Hitchin base} Let $(V,\phi)$ be a Higgs bundle on a complete, non-singular scheme $Y$, and let  $r\geqslant 1$ be the rank of $V$. The characteristic polynomial  $\det(T-\phi)$ of the twisted endomorphism $\phi\in H^0(\mathit{End}(V)\otimes\Omega_Y)$ has the form
$$
T^r+a_1T^{r-1}+\cdots+a_{r-1}T+a_r,
$$
with $a_i\in H^0(\mathrm{Sym}^i\Omega_Y)$. The affine space $\mathcal{B}_{Y,r}$   associated to the vector space
$$
\bigoplus_{i=1}^r H^0(\mathrm{Sym}^i\Omega_Y)
$$
is called the \emph{Hitchin base} (for rank $r$ Higgs bundles on $Y$). Let $M^{\mathrm{Higgs}}_{Y,r}$ be the moduli space of rank $r$, semistable Higgs bundles on $Y$ \cite{simp 3}. Then the association $(V,\phi)\mapsto(a_1,\dots,a_r)$ induces a morphism 
$$
M^{\mathrm{Higgs}}_{Y,r}\to \mathcal{B}_{Y,r},
$$
called the \emph{Hitchin fibration}. It is a useful tool for the study of $M^{\mathrm{Higgs}}_{Y,r}$.

Our first result concerning Higgs bundles on elliptic surfaces is the following Proposition, which suggests a strong relation between Higgs bundles on the total space and on the base of an elliptic surface, and whose proof follows immediately from Proposition \ref{le potenze}:

\begin{Prop}
\label{la base}
Let $\pi:X\to B$ be a non-isotrivial Weierstrass fibration with nodal singular fibers. Then, for each $r\geqslant 1$, there is a canonical isomorphism
\begin{equation}
\label{iso basi}
\pi^\ast:\mathcal{B}_{B,r}\to\mathcal{B}_{X,r}.
\end{equation}
\end{Prop}

Let us remark that, if we use the isomorphism \eqref{iso basi} to identify $\mathcal{B}_{X,r}$ with $\mathcal{B}_{B,r}$, we have that the operation of pull-back via $\pi$ induces a rational map of $\mathcal{B}_{B,r}$-schemes 
$$
M_{B,r}^{\mathrm{Higgs}}\dashrightarrow M_{X,r}^{\mathrm{Higgs}}.
$$
Moreover, using again Proposition \ref{le potenze}, we get an isomorphism of $\mathcal{O}_B$-algebras
$$
\pi_\ast\mathrm{Sym}\,\Theta_X\simeq\mathrm{Sym}\,\mathbb{L}^{-1};
$$
thus, the operation of push-forward via $\pi$ induces a rational map from the moduli space of semistable Higgs bundles on $X$ to the moduli space of semistable $\mathbb{L}$-valued pairs on $B$ (see, e.g.,  \cite{nitsure} for the notions of semistable pairs and their moduli).

\subsection{Vertical Higgs fields}

From now on, let $(\pi:X\to B,\Sigma)$ be a non-isotrivial Weierstrass fibration with nodal singular fibers. Let $V$ be a vector bundle on $X$. Then, starting from a linear map $\psi:V\to V\otimes\pi^\ast\omega_B$, we get a linear map $\phi:V\to V\otimes\Omega_X$, by composing $\psi$ on the left with
$$
1\otimes\pi^\ast:V\otimes\pi^\ast\omega_B\to V\otimes\Omega_X.
$$
Moreover, the map $\phi$ obtained in this way satisfies the integrability condition, i.e., it is a Higgs field on $V$; and the correspondence $\psi\mapsto\phi$ is injective, since it can be regarded as the map on global sections induced by the injection
\begin{equation}
\label{lisomorphidue}
1\otimes\pi^\ast:\mathit{End}(V)\otimes\pi^\ast\omega_B\to\mathit{End}(V)\otimes\Omega_X.
\end{equation}
It follows that any bundle $V$ on an elliptic fibration admits a family of Higgs fields parametrized by the vector space $\mathrm{Hom}(V,V\otimes\pi^\ast\omega_B)$, which we will refer to as \emph{vertical} Higgs fields. We now show that, under suitable assumptions on $V$, there are no other Higgs fields. More precisely, we have the following:

\begin{Prop}
\label{propo}
Let $V$ be a rank $r$ vector bundle on $X$ with vertical determinant, and let $E$ be its sheaf of endomorphisms. Suppose that $V$ satisfies the following two assumptions:
\begin{enumerate}
\item[(i)] $V$ is semistable on a general closed fiber of $\pi$;
\item[(ii)] the spectral cover of $V$ is reduced.
\end{enumerate}
Then the natural map
\begin{equation}
\label{lisomorphismo}
\pi_\ast(1\otimes\pi^\ast):\pi_\ast(E\otimes\pi^\ast\omega_B)\to \pi_\ast(E\otimes\Omega_X)
\end{equation}
is an isomorphism of rank $r$ bundles on $B$.
\end{Prop}
 
This immediately implies the following: 
 
\begin{corollary}
\label{corollll}
Let $V$ be as in Proposition \ref{propo}. Then the only Higgs fields on $V$ are vertical.
\end{corollary}

\begin{proof}
It follows immediately from Proposition \ref{propo} by taking the map on global sections induced by the isomorphism \eqref{lisomorphismo}, and noting that this is the same as the map on global sections induced by \eqref{lisomorphidue}.
\end{proof}

\begin{proof}[Proof of Proposition \ref{propo}]
Tensoring \eqref{exact sequence of kahler differntials} by the locally free sheaf $E$, we get the short exact sequence of torsion free sheaves on $X$
$$
E\otimes \mathcal{E}:\quad 0\to E\otimes\pi^\ast\omega_B\to E\otimes\Omega_X\to E\otimes\Omega_\pi\to 0;
$$
taking direct images, we then get the exact sequence of (locally free, by Lemma \ref{lemma utile}) sheaves on $B$
\begin{equation}
\label{la successione}
0\to\pi_\ast( E\otimes\pi^\ast\omega_B)\to\pi_\ast(E\otimes\Omega_X)\to\pi_\ast(E\otimes\Omega_\pi).
\end{equation}
{The} reducedness of the spectral cover implies that, on a general closed fiber $X_b$ of $\pi$, the bundle $V_b$ is isomorphic to a direct sum of $r$ pairwise distinct line bundles of degree zero on $X_b$, say
$$
V_b\simeq\lambda_1\oplus\cdots\oplus\lambda_r.
$$
This, together with the fact that $\pi^\ast\omega_B$ and $\Omega_\pi$ are trivial on a general closed fiber, implies immediately that $\pi_\ast( E\otimes\pi^\ast\omega_B)$ and $\pi_\ast(E\otimes\Omega_\pi)$ have rank $r$. Finally, using the fact that $\Omega_{X,b}$ is isomorphic to the rank $2$ Atiyah bundle on $X_b$ for $b\in B(k)$ general, we find
$$
(E\otimes\Omega_X)_b\simeq\mathit\bigoplus_{i,j=1}^r\lambda_i\otimes\lambda_j^{-1}\otimes\mathrm{I}_2;
$$
thus, 
$$
 h^0(X_b,(E\otimes\Omega_X)_b)=\sum_{i,j=1}^r h^0(X_b,\lambda_i\otimes\lambda_j^{-1}\otimes\mathrm{I}_2)=\sum_{i,j=1}^r\delta_{ij}=r,
$$
which shows that $\pi_\ast(E\otimes\Omega_X)$ has rank $r$ too. It follows that \eqref{la successione} is an exact sequence of rank $r$ bundles, and this implies immediately that the first map is an isomorphism, as claimed.
\end{proof}

\subsubsection{Remarks on the assumptions of Proposition \ref{propo}}We remark that the assumptions in \ref{propo} are both necessary. To show that assumption (i) is necessary, set $\ell:=\mathcal{O}_X(\Sigma)$, $V_1=\ell\oplus \ell^{-1}$. Then $V_1$ is a rank $2$ bundle with trivial determinant, but it is unstable on each fiber; moreover $
E_1:=\mathit{End}(V_1)$ satisfies $E_1\simeq \ell^{ 2}\oplus\mathcal{O}_X^{\oplus 2}\oplus \ell^{-2}$. Thus by Proposition \ref{immagini multipli sezione}
$$
\pi_\ast( E_1\otimes\pi^\ast\omega_B)\simeq \omega_B\otimes(\mathcal{O}_B^{\oplus 3}\oplus\mathbb{L}^{-2})
$$
has rank $4$; instead, $\pi_\ast(E_1\otimes\Omega_X)$ has rank $6$, since for $b\in B(k)$ general, we have, setting $p:=\sigma(b)\in X_b$,
$$
 h^0((E_1\otimes\Omega_X)_b)= h^0((\mathcal{O}(2p)\oplus\mathcal{O}^{\oplus 2}\oplus\mathcal{O}(-2p))\otimes\mathrm{I}_2)=4+2+0.
$$

To show that assumption (ii) is necessary, consider a rank $r\geqslant 2$ bundle $V_2$ on $X$ whose restriction to a general fiber is isomorphic to the rank $r$ Atiyah bundle $\mathrm{I}_r$ (such bundles exist; for  $r=2$, one could take, e.g., $V_2=\Omega_X$), so that $V_2$ is semistable (in fact, regular) and has trivial determinant on a general fiber, but it has non-reduced spectral cover $r\Sigma$. Let $E_2:=\mathit{End}(V_2)$. Then $\pi_\ast(E_2\otimes\pi^\ast\omega_B)$ has rank $r$, while $\pi_\ast(E_2\otimes\Omega_X)$ has rank $2r-1\neq  r$.

\subsection{Scalar Higgs fields}
\label{scalar Higgs}
Let $Y$ be a non-singular, complete scheme; let $V$ be a locally free sheaf on $Y$ of rank $r>0$, and $E$ its sheaf of endomorphisms. The linear map $\mathcal{O}_Y\to E$ corresponding to the identity section of $E$ is injective; thus, tensoring by $\Omega_Y$, and then taking global sections,  we get an injective $k$-linear map
$$
 H^0(\Omega_Y)\hookrightarrow H^0(E\otimes\Omega_Y),\qquad\alpha\mapsto\phi_\alpha,
$$
which factors through the inclusion $\mathcal{H}_V\hookrightarrow  H^0(E\otimes\Omega_Y)$. In fact, for each global $1$-form $\alpha$ on $Y$, the field $\phi_\alpha$ acts on a local section $s$ of $V$, defined on a Zariski open $\mathcal{U}$ of $Y$, as $\phi_\alpha(s)=s\otimes\alpha|_\mathcal{U}$; from this it follows immediately that $\phi_\alpha$ satisfies the integrability condition $\phi_\alpha\wedge\phi_\alpha=0$, and that, moreover, the matrix of $\phi_\alpha$ with respect to a local frame $(e_1,\dots,e_r):\mathcal{O}_\mathcal{U}^{\oplus r}\xrightarrow{\simeq} V|_\mathcal{U}$ for $V$
is the (scalar) matrix of $1$-forms on $\mathcal{U}$
$$
\mathrm{diag}(\alpha|_\mathcal{U},\dots,\alpha|_\mathcal{U});
$$
thus the Higgs fields on $V$ obtained  in this way might be called \emph{scalar} Higgs fields. This shows that $V$ admits a family of Higgs fields of dimension $ h^0(\Omega_Y)$. For example, for a Weierstrass fibration $\pi:X\to B$ (satisfying the usual assumptions), one has, using the case $r=1$ of Proposition \ref{le potenze},
$$
 h^0(X,\Omega_X)= h^0(B,\pi_\ast\Omega_X)= h^0(B,\omega_B)=g,
$$
where $g$ is the genus of the base curve $B$.

We remark that scalar Higgs fields pull-back to scalar Higgs fields; more precisely, for a morphism $f:Z\to Y$, and a vector bundle $V$ on $Y$, one has a commutative diagram
\begin{equation}\label{scalardiag}
\begin{tikzcd}
 H^0(\Omega_Y)\arrow[hook]{r}\arrow{d}&\mathcal{H}_V
\arrow{d}\\
 H^0(\Omega_Z)
\arrow[hook]{r}& 
\mathcal{H}_{f^\ast V}
\end{tikzcd}
\end{equation}
where the vertical arrows are the pull-back maps on global $1$-forms and on Higgs fields.

 Our next goal is to show that, under suitable assumptions, a vector bundle on a Weierstrass fibration supports only scalar Higgs fields.

\subsubsection{The Universal Spectral Cover}

We fix a Weierstrass fibration $(\pi:X\to B,\Sigma)$, and an integer $r\geqslant 2$. We recall, from Section \ref{uno}, that $\mathcal{S}=\pi_\ast\mathcal{O}_X(r\Sigma)$ is a rank $r$ vector bundle on $B$. The morphism $p:\mathbb{P}_B(\mathcal{S})\to B$ is then a $\mathbb{P}^{r-1}$-bundle on $B$, while $\tilde{p}:\mathbb{P}_X(\pi^\ast \mathcal{S})\to X$ is a $\mathbb{P}^{r-1}$-bundle on $X$; and there is a canonical morphism $\tilde{\pi}:\mathbb{P}_X(\pi^\ast \mathcal{S})\to\mathbb{P}_B(\mathcal{S})$ such that
$$
\tilde{\pi}^\ast\mathcal{O}_{\mathbb{P}_B(\mathcal{S})}(1)\simeq\mathcal{O}_{\mathbb{P}_X(\pi^\ast \mathcal{S})}(1),
$$
and the square
$$
\begin{tikzcd}
\mathbb{P}_X(\pi^\ast \mathcal{S})\arrow{r}{\tilde{\pi}}\arrow{d}{\tilde{p}}
               &\mathbb{P}_B(\mathcal{S})\arrow{d}{p}\\
X\arrow{r}{\pi}& B
\end{tikzcd}
$$
is cartesian. Now let $\mathcal{K}$ be the kernel of the canonical epimorphism $\pi^\ast \mathcal{S}\twoheadrightarrow\mathcal{O}_X(r\Sigma)$. Then $\mathcal{K}$ is a vector bundle on $X$ of rank $r-1$; 
thus $\mathcal{C}:=\mathbb{P}_X(\mathcal{K})$ is a $\mathbb{P}^{r-2}$-bundle on $X$, which is the relative incidence correspondence of the family of complete linear systems on the fibers of $\pi$ associated to the line bundle $\mathcal{O}_X(r\Sigma)$. It is a smooth, prime divisor of $\mathbb{P}_X(\pi^\ast \mathcal{S})$, whose associated invertible sheaf $\mathcal{O}_{\mathbb{P}_X(\pi^\ast \mathcal{S})}(\mathcal{C})$ is given by
\begin{eqnarray*}
\mathcal{O}_{\mathbb{P}_X(\pi^\ast \mathcal{S})}(\mathcal{C}) &\simeq &\mathcal{O}_{\mathbb{P}_X(\pi^\ast \mathcal{S})}(1)\otimes\tilde{p}^\ast\mathcal{O}_X(r\Sigma)\\
&\simeq &\tilde{\pi}^\ast\mathcal{O}_{\mathbb{P}_B(\mathcal{S})}(1)\otimes\tilde{p}^\ast\mathcal{O}_X(r\Sigma).
\end{eqnarray*}
Moreover, $\tilde{p}$ restricts to $\mathcal{C}$ to give the bundle projection $\mathbb{P}_X(\mathcal{K})\to X$, while the restriction of $\tilde{\pi}$ to $\mathcal{C}$ is a finite, degree $r$ morphism
$$
\varphi:\mathcal{C}\to\mathbb{P}_B(\mathcal{S}).
$$
The following commutative diagram summarizes the situation:
$$
\begin{tikzcd}
 & &X\arrow{rd}{\pi}&\\
\mathcal{C}\arrow[bend left]{rru}\arrow[hook]{r}\arrow[bend right]{rrd}{}[swap]{\varphi}&\mathbb{P}_X(\pi^\ast \mathcal{S})\arrow{ru}{\tilde{p}}\arrow{rd}{}[swap]{\tilde{\pi}}&&B\\
 & &\mathbb{P}_B(\mathcal{S})\arrow{ru}{}[swap]{p}& \\
\end{tikzcd}.
$$
The morphism $\varphi$ is called the \emph{universal spectral cover}, for the following reason: let $V$ be a rank $r$ bundle on $X$, with vertical determinant and semistable on a general fiber of $\pi$. We saw that the map $\det(\mathrm{ev}:\pi^\ast\pi_\ast V(\Sigma)\to V(\Sigma))$ can be considered as a section of $\mathcal{S}\otimes\mu$, for a suitable $\mu\in\mathrm{Pic}(B)$, or as a morphism $\mathcal{S}^\vee\to\mu$, surjective at the generic point of $B$. It thus gives rise to a rational section of $p:\mathbb{P}_B(\mathcal{S})\to B$, which then extends to a global section $A_V:B\to\mathbb{P}_B(\mathcal{S})$ (since $B$ is a smooth curve, and $\mathbb{P}_B(\mathcal{S})$ is complete). Then one can show that the spectral cover $\varphi_V:C_V\to B$ of $V$ is isomorphic (as a $B$-scheme) to the base change of the universal spectral cover via the section $A_V$; in other words, there exists a closed immersion $C_V\hookrightarrow\mathcal{C}$ such that the square
$$
\begin{tikzcd}
C_V\arrow[hook]{r}\arrow{d}{}[swap]{\varphi_V}
               &\mathcal{\mathcal{C}}\arrow{d}{\varphi}\\
B\arrow{r}{A_V}& \mathbb{P}_B(\mathcal{S})
\end{tikzcd}
$$
is cartesian.

We will use the following:

\begin{lemma}
\label{il lemma bla bla bla}
There is an exact sequence of vector bundles on $\mathbb{P}_B(\mathcal{S})$:
\begin{equation}
\label{successione esatta}
0\to\mathcal{O}_{\mathbb{P}_B(\mathcal{S})}\to\varphi_\ast\mathcal{O}_\mathcal{C}\to\mathcal{O}_{\mathbb{P}_B(\mathcal{S})}(-1)\otimes p^\ast( \mathcal{S}^\vee\otimes\mathbb{L}^{-1}).
\end{equation}
\end{lemma}

\begin{proof}
Set $\mathbb{P}:=\mathbb{P}_B(\mathcal{S})$, $\mathbb{P}^\ast:=\mathbb{P}_X(\pi^\ast \mathcal{S})$. We have the short exact sequence of sheaves on $\mathbb{P}^\ast$:
$$
0\to\mathcal{O}_{\mathbb{P}^\ast}(-\mathcal{C})\to\mathcal{O}_{\mathbb{P}^\ast}\to\mathcal{O}_\mathcal{C}\to 0.
$$
Taking direct $\tilde{\pi}$-images we get the exact sequence of sheaves on $\mathbb{P}$:
\begin{equation}
\label{successione del lemma}
0\to\tilde{\pi}_\ast\mathcal{O}_{\mathbb{P}^\ast}(-\mathcal{C})\to\tilde{\pi}_\ast\mathcal{O}_{\mathbb{P}^\ast}\to\tilde{\pi}_\ast\mathcal{O}_\mathcal{C}\to\mathrm{R}^1\tilde{\pi}_\ast\mathcal{O}_{\mathbb{P}^\ast}(-\mathcal{C}).
\end{equation}
Let us compute the sheaves in the last sequence. We have
\begin{eqnarray*}
\mathrm{R}^i\tilde{\pi}_\ast
\mathcal{O}_{\mathbb{P}^\ast}(-\mathcal{C})
&\simeq &
\mathrm{R}^i\tilde{\pi}_\ast
(
\tilde{\pi}^\ast\mathcal{O}_{\mathbb{P}}(-1)\otimes\tilde{p}^\ast\mathcal{O}_X(-r\Sigma)
)\\
&\simeq &
\mathcal{O}_{\mathbb{P}}(-1)\otimes\mathrm{R}^i\tilde{\pi}_\ast
(
\tilde{p}^\ast\mathcal{O}_X(-r\Sigma)
)\\
&\simeq &
\mathcal{O}_{\mathbb{P}}(-1)\otimes
p^\ast\mathrm{R}^i\pi_\ast
\mathcal{O}_X(-r\Sigma),
\end{eqnarray*}
where the sheaf $\mathrm{R}^i\pi_\ast
\mathcal{O}_X(-r\Sigma)$ is zero for $i=0$, while for $i=1$ it is isomorphic to
\begin{eqnarray*}
\mathit{Ext}^0_\pi(\mathcal{O}_X(-r\Sigma),\omega_\pi^\circ)^\vee &\simeq &
(
\pi_\ast(\mathcal{O}_X(r\Sigma)\otimes\pi^\ast\mathbb{L})
)^\vee\\
&\simeq & \mathcal{S}^\vee\otimes\mathbb{L}^{-1};
\end{eqnarray*}
it follows that
$$
\tilde{\pi}_\ast\mathcal{O}_{\mathbb{P}^\ast}(-\mathcal{C})=0,\qquad \mathrm{R}^1\tilde{\pi}_\ast\mathcal{O}_{\mathbb{P}^\ast}(-\mathcal{C})\simeq 
\mathcal{O}_{\mathbb{P}}(-1)\otimes
p^\ast( \mathcal{S}^\vee\otimes\mathbb{L}^{-1}).
$$
Finally, we have
$$
\tilde{\pi}_\ast\mathcal{O}_{\mathbb{P}^\ast}\simeq
\tilde{\pi}_\ast\tilde{p}^\ast\mathcal{O}_X\simeq p^\ast\pi_\ast\mathcal{O}_X\simeq p^\ast\mathcal{O}_B\simeq\mathcal{O}_{\mathbb{P}},
$$
and, denoting by $\iota:\mathcal{C}\hookrightarrow\mathbb{P}^\ast$ the inclusion,
$$
\tilde{\pi}_\ast\mathcal{O}_\mathcal{C}=\tilde{\pi}_\ast\iota_\ast\mathcal{O}_\mathcal{C}=(\tilde{\pi}\circ\iota)_\ast\mathcal{O}_\mathcal{C}=\varphi_\ast\mathcal{O}_\mathcal{C}.
$$
Substituting in \eqref{successione del lemma}, we get the sequence \eqref{successione esatta}.
\end{proof}

\subsubsection{}

Using lemma \ref{il lemma bla bla bla}, we can prove the following proposition, where, as usual, we denote by $(\pi:X\to B,\Sigma)$ a non-isotrivial Weierstrass fibration (with smooth base and total space) with nodal singular fibers; moreover, $d$ is the degree of the fundamental line bundle $\mathbb{L}$ of $(\pi,\Sigma)$, and $g$ the genus of the (complete, irreducible) curve $B$. 

\begin{Prop}
\label{uno dei risultati}
Let $V$ be a vector bundle on $X$ of rank $r\geqslant 2$, which is fiberwise \emph{regular} and has vertical determinant. Let $\varphi_V:C_V\to B$ be the spectral cover of $V$, and let us assume that at least one of the following two assumptions is satisfied:
\begin{enumerate}
\item[(i)] $C_V$ is integral and $d\geqslant 2g-1$;
\item[(ii)] $C_V$ is reduced and $\mathrm{c}_2(V)\geqslant (r-1)d+2g-1$.
\end{enumerate}
Then   {every} Higgs field on $V$ {is} scalar.
\end{Prop}

\begin{proof}
Let us start by remarking that, under either of the assumptions (i) or (ii), we can apply to $V$ Proposition \ref{propo} and its Corollary \ref{corollll}. Thus, any map $V\to V\otimes\Omega_X$ is a Higgs field, and these are the same as the global sections of the vector bundle on $B$
$$
\pi_\ast(\pi^\ast\omega_B\otimes\mathit{End}V)\simeq\omega_B\otimes\pi_\ast\mathit{End}\,V.
$$
Moreover, since $V$ is fiberwise regular, the bundle $\pi_\ast\mathit{End}\,V$ is isomorphic to $\varphi_{V\ast}\mathcal{O}_{C_V}$ by Proposition \ref{sui fib reg}. 

Set $\mathcal{S}:=\pi_\ast\mathcal{O}_X(r\Sigma)\simeq\bigoplus_{i\in\{2,\dots,r,0\}}\mathbb{L}^{-i}$, and denote by $p:\mathbb{P}\to B$ the projectivization of $\mathcal{S}$, and by $A_V:B\to\mathbb{P}$ the section of $p$ corresponding to $V$; recall that we have a cartesian square with finite vertical arrows 
$$
\begin{tikzcd}
C\arrow[hook]{r}\arrow{d}{}[swap]{\varphi_V}
               &\mathcal{C}\arrow{d}{\varphi}[swap]{}\\
B\arrow[hook]{r}{A_V}& \mathbb{P}
\end{tikzcd},
$$
where $\varphi:\mathcal{C}\to\mathbb{P}$ is the universal spectral cover. Thus, for each sheaf $F$ on $\mathcal{C}$, the natural map
$$
A_V^\ast\varphi_\ast F\to\varphi_{V\ast}(F\otimes\mathcal{O}_C)
$$
is an isomorphism (see, e.g., Lemma $5.6$ of \cite{mfw}). In particular, taking $F=\mathcal{O}_\mathcal{C}$, we get an isomorphism
$$
A_V^\ast\varphi_\ast\mathcal{O}_\mathcal{C}\simeq \varphi_{V\ast}\mathcal{O}_{C_V}.
$$
It follows that if we pull the exact sequence \eqref{successione esatta} back to $B$ using $A_V:B\to\mathbb{P}$, we get an exact sequence of vector bundles on $B$
$$
0\to\mathcal{O}_B\to\varphi_{V\ast}\mathcal{O}_{C_V}\to\mu^{-1}\otimes \mathcal{S}^\vee\otimes\mathbb{L}^{-1},
$$
where $\mu=A_V^\ast\mathcal{O}_\mathbb{P}(1)$ is the line bundle \eqref{def of mu}; tensoring with $\omega_B$ we obtain an exact sequence
\begin{equation}
\label{successione intermedia}
0\to\omega_B\to\omega_B\otimes\varphi_{V\ast}\mathcal{O}_{C_V}\to \mathcal{S}^\vee\otimes\mathbb{L}^{-1}\otimes\mu^{-1}\otimes\omega_B.
\end{equation}
The last bundle splits as $\lambda_2\oplus\cdots\oplus\lambda_r\oplus\lambda_0$, where
$$
\lambda_i:=\mathbb{L}^{i-1}\otimes\mu^{-1}\otimes\omega_B\in\mathrm{Pic}(B).
$$
Now recall that $d\geqslant 1$, and that $e:=\deg(\mu)$ satisfies $e=\mathrm{c}_2(V)$ by Lemma \ref{e uguale c2},  and the inequalities
$$
e\geqslant
\left\{
\begin{array}{ccc}
rd & , &\mbox{if }C_V\mbox{ is integral}\\
(r-1)d & , &\mbox{if }C_V\mbox{ is reduced}
\end{array}
\right.
$$
by Lemma \ref{le prop di spec cover}. It follows that, for each $i\in\{0,2,\dots,r\}$, the degree
$$
\deg(\lambda_i)=(i-1)d-e+2g-2
$$
satisfies
$$
\deg(\lambda_i)\leqslant (r-1)d-e+2g-2,
$$
and the right hand side of this inequality is negative under any of the assumptions (i) and (ii). Thus
$$
 H^0(\mathcal{S}^\vee\otimes\mathbb{L}^{-1}\otimes\mu^{-1}\otimes\omega_B)=\bigoplus_i H^0(\lambda_i)=0,
$$
and the arrow $\omega_B\to\omega_B\otimes\varphi_{V\ast}\mathcal{O}_{C_V}$ in \eqref{successione intermedia} induces an isomorphism on global sections. It follows that
the natural map $ H^0(\Omega_X)\hookrightarrow\mathcal{H}_V=\mathrm{Hom}(V,V\otimes\Omega_X)$ is an injection between vector spaces of the same dimension $g$, and thus an isomorphism.
\end{proof}

Proposition \ref{uno dei risultati}  is useful in conjunction with the following sufficient condition for the base-point freeness of linear systems on $X$ of the form $|r\Sigma+\pi^\ast\mu|$ ($r\in\mathbb{Z}$, $\mu\in\mathrm{Pic}(B)$):

\begin{Prop}
\label{base point free}
Let $r$ be an integer $\geqslant 2$, and let $\mu$ be a line bundle on $B$ of degree $e\geqslant rd+2g$. Then the linear system on $X$
$$
|r\Sigma+\pi^\ast\mu|
$$
is base-point free.
\end{Prop}

\begin{proof}
For each $b\in B(k)$, the line bundle $\lambda:=\mathcal{O}_X(r\Sigma)\otimes\pi^\ast\mu$ on $X$ restricts on $X_b$ to $\mathcal{O}_{X_b}(r\sigma(b))\in\mathrm{Pic}(X_b)$, and the linear system $|r\sigma(b)|$ on $X_b$ is base-point free, since $r$ is assumed to be $\geqslant 2$. It is then enough to check that the restriction map
$$
\mathrm{res}: H^0(X,r\Sigma+\pi^\ast\mu)\to H^0(X_b,r\sigma(b))
$$
is surjective. Its cokernel injects into $
 H^1(X,\lambda(-X_b))= H^1(X,r\Sigma+\pi^\ast(\mu-b))$;
by the Leray spectral sequence we have
$$
 h^1(X,\lambda(-X_b))= h^1(B,\mathrm{R}^0\pi_\ast(\lambda(-X_b)))+ h^0(B,\mathrm{R}^1\pi_\ast(\lambda(-X_b))),
$$
with $
\mathrm{R}^i\pi_\ast(\lambda(-X_b))=\mathrm{R}^i\pi_\ast\mathcal{O}_X(r\Sigma)\otimes\mu(-b)$; the last sheaf is zero for $i=1$, while for $i=0$ it is isomorphic to $\bigoplus_{j\in\{2,\dots,r,0\}}\mathbb{L}^{-j}\mu(-b)$  by Proposition \ref{immagini multipli sezione}(ii) (here and in the following we omit some $\otimes$'s for simplicity of notation). Thus we find
$$
 h^1(X,\lambda(-X_b))=\sum_j h^1(\mathbb{L}^{-j}\mu(-b))=\sum_j h^0(\omega_B\mathbb{L}^j\mu^{-1}\otimes\mathcal{O}_B(b)),
$$
and each of the summands in the last sum is zero, since
$$
\deg(\omega_B\mathbb{L}^j\mu^{-1}\otimes\mathcal{O}_B(b))=2g-2+jd-e+1\leqslant rd+2g-e-1\leqslant -1<0.
$$
This shows that $\mathrm{res}$ is surjective, as claimed.
\end{proof}

\medskip
\section{Sheaves with vanishing discriminant}
\label{ultima sezione}
\subsection{Stability}

Let us recall that the notions of degree, slope, and slope-semistability for a sheaf on a projective scheme $Y$ require the choice of an ample divisor $H$ on $Y$ (or of an open ray $\mathbb{R}_{>0}\cdot H$ in the ample cone of $Y$), which in this context is called a \emph{polarization} on $Y$. A \emph{polarized variety} is then a pair $(Y,H)$, where $Y$ is a non-singular, irreducible, projective scheme, and $H$ a polarization on $Y$.

Let $(Y,H)$ be a polarized variety, with $n:=\dim(Y)\geqslant 1$. The \emph{degree} of a coherent sheaf $F$ on $Y$ is the integer
 $$
 \deg(F):=\mathrm{c}_1(F)\cdot  H^{n-1};
 $$
 if $F$ has positive rank, the \emph{slope} of $F$ is the rational integer
 $$
 \mu(F):=\frac{\deg(F)}{\mathrm{rk}(F)}.
 $$
 A torsion-free sheaf $F$ on $Y$ is said to be \emph{slope-semistable} if, for each non-zero subsheaf $S$ of $F$, one has
 \begin{equation}
 \label{stab cond}
 \mu(S)\leqslant\mu(F)
 \end{equation}
 (or, equivalently, if $\mu(F)\leqslant\mu(Q)$ for each quotient $Q$ of $F$ with $\mathrm{rk}(Q)\neq 0$). Since this is the only notion of stability that we shall consider, we will sometimes shorten slope-semistable to \emph{semistable}.

\subsection{The discriminant and Bogomolov inequality} For the reader's convenience this and the next subsection will repeat some definitions and statements already given in the Introduction.
Let $Y$ be an irreducible, non-singular, projective scheme, and let $F$ be a rank $r$ sheaf on $Y$. The \emph{discriminant} of $F$ is the characteristic class
$$
\Delta(F):=2r\mathrm{c}_2(F)-(r-1)\mathrm{c}_1(F)^2\in\mathrm{A}^2(Y).
$$
(here $\mathrm{A}^2(Y)$ is the Chow group of codimension $2$ cycles on $Y$ modulo rational equivalence) If $F$ is locally free, this is the same as $\mathrm{c}_2(\mathit{End}(F))=\mathrm{c}_2(F^\vee\otimes F)$, whose expansion in terms of the Chern roots of $F$ might help explaining the choice of the name discriminant. If $Y$ has dimension $n\geqslant 2$ and is polarized by the ample divisor $H$, one can multiply the discriminant of $F$ by a suitable power of $H$, to obtain an integer, called the \emph{Bogomolov number} of $F$ (with respect to $H$), and denoted by $\mathrm{B}(F)$:
$$
\mathrm{B}(F):=\Delta(F)\cdot  H^{n-2}.
$$
Of course {when $Y$ is a surface  $\mathrm{B}(F)$} is  independent {of} the polarization, and it is just the image of $\Delta(F)$ in $\mathbb{Z}$ under the degree homomorphism $\mathrm{A}^2(Y)=\mathrm{A}_0(Y)\to \mathbb{Z}$.

The Bogomolov number is the subject of a remarkable theorem, the \emph{Bogomolov inequality}, proved for the first time in \cite{bog} for a vector bundle on a complex surface (see \cite{langer} for a proof of the general statement):

\begin{thm}[Bogomolov inequality]
Let $(Y,H)$ be a polarized variety, and let $F$ be a torsion-free coherent sheaf on $Y$. Then, if $F$ is slope-semistable, the Bogomolov number of $F$ is non-negative.
\end{thm}

It is then natural to try to find a characterization of the slope-semistable sheaves with vanishing Bogomolov number. The main result in this direction is the following \cite{nakayama, bruzzo herna} (this is Theorem \ref{caratt} in the Introduction):

\begin{thm}
\label{caratterizzazione dei bandoli con...}
Let $(Y,H)$ be a \emph{complex} polarized variety. Then, for a locally free sheaf $F$ on $Y$, the following are equivalent:
\begin{enumerate}
\item[(i)] $F$ is slope-semistable and has vanishing discriminant {in $H^4(Y,\mathbb Q)$};
\item[(ii)] for every irreducible, non-singular, projective curve $C$, and for every morphism $f:C\to Y$, the pull-back $f^\ast F$ of $F$ to $C$ via $f$ is semi-stable.
\end{enumerate}
\end{thm}

Following \cite{bruzzo valeriano}, we call a sheaf $F$ satisfying (ii) of Theorem \ref{caratterizzazione dei bandoli con...} \emph{curve-semistable}. Then Theorem \ref{caratterizzazione dei bandoli con...} can be paraphrased by saying that, on a complex polarized variety, the semistable locally free sheaves with vanishing discriminant are exactly the locally free sheaves that are curve-semistable.

\subsection{The Higgs case and the conjecture}

There is a notion of slope-semistabilty adapted to torsion-free Higgs sheaves on a polarized variety $(Y,H)$:
a torsion-free Higgs sheaf $(F,\phi)$ on $Y$ is slope-semistable (with respect to $H$) if it satisfies the inequality \eqref{stab cond}  for {the non-zero subsheaves $S$ of $F$ that} are \emph{invariant} under the action of the Higgs field $\phi$ (i.e., such that the restriction of $\phi$ to $S$ factors through the injection $(S\hookrightarrow F)\otimes\Omega_Y$). Then one has a version of Bogomolov inequality for Higgs sheaves (proved in \cite{simp 1} for \emph{stable} Higgs bundles over \emph{complex} varieties, and in \cite{langer} in the general case):

\begin{thm}[Bogomolov inequality for Higgs sheaves]
Let $(Y,H)$ be a polarized variety, and let $(F,\phi)$ be a torsion-free Higgs sheaf on $Y$. Then, if $(F,\phi)$ is slope-semistable, the Bogomolov number of $F$ is non-negative.
\end{thm}

Once again, it is then natural to try to characterize the torsion-free, slope-semistable Higgs sheaves with vanishing Bogomolov number. In view of Theorem \ref{caratterizzazione dei bandoli con...}, the following statement (this is Conjecture \ref{conjec} from the Introduction) appears to be natural. In it we use the notion of curve-semistability for a Higgs sheaf $(F,\phi)$ on a polarized variety $(Y,H)$: $(F,\phi)$ is \emph{curve-semistable} if, for each pair $(C,f)$, where $C$ is an irreducible, non-singular, projective curve, and $f:C\to Y$ a morphism, the pull-back Higgs sheaf $f^\ast(F,\phi)$ is semistable.

\begin{conj}
\label{conjecture}
Let $(Y,H)$ be a \emph{complex} polarized variety, and let $(F,\phi)$ be a Higgs bundle on $Y$. Then the following are equivalent:
\begin{enumerate}
\item[(i)] $(F,\phi)$ is semistable with vanishing discriminant {in $H^4(Y,\mathbb Q)$};
\item[(ii)] $(F,\phi)$ is curve-semistable.
\end{enumerate}
\end{conj}

The implication (i) $\Rightarrow$ (ii) of this conjecture {was} proved in \cite{bruzzo beatriz, bruzzo herna}. Moreover, the Higgs version of  the Mehta-Ramanathan theorem \cite{simp 2} implies that a curve-semistable Higgs bundle is semistable. So, what is left to be proved is the statement that a curve-semistable Higgs bundle has vanishing discriminant, or, equivalently, that a Higgs bundle with non-zero discriminant is unstable (i.e., non-semistable) when pulled back to a suitable curve.

{ It has been proved over the last few years} that the conjecture is true for some classes of varieties, including those with nef tangent bundle \cite{bruzzo alessio}, K3 surfaces \cite{bruzzo valeriano}, and, more generally, Calabi-Yau varieties \cite{bruzzo capasso}. Thus, if we restrict ourselves to the case of surfaces, we can consider the conjecture as proved for surfaces of Kodaira dimension $\leqslant 0$. The next case which is then natural to examine is that of surfaces of Kodaira dimension $1$, the so-called honest elliptic surfaces. In the next section, we will use the results on Higgs bundles proved in Section \ref{3} to make some progress in the study of the conjecture in the case of elliptic surfaces.

\subsection{Study of the conjecture on elliptic surfaces}

Assume $k=\mathbb{C}$ in this section. Let $(\pi:X\to B,\Sigma)$ be a non-isotrivial Weierstrass fibration with nodal singular fibers. As always, we denote by   $\mathbb{L}\in\mathrm{Pic}^d(B)$ ($d\geqslant 1$) the fundamental line bundle of $X$, and by $g$ the genus of $B$. Fix a polarization $H$ on $X$.
For each Higgs bundle $\mathcal{V}=(V,\phi)$ on $X$, let us denote by $\mathcal{P}_{\mathcal{V}}$ the following claim: 
\begin{center}
\emph{if $\mathcal V$ is semistable, and $\Delta(V)$ is non-zero, then there exists a pair $(C,f)$, where $C$ is a smooth, irreducible, projective curve, and $f:C\to X$ a morphism, such that the pull-back $f^\ast\mathcal{V}$ is unstable}.
\end{center}
 Then, by our previous remarks, Conjecture \ref{conjecture} holds if and only if claim $\mathcal{P}_\mathcal{V}$ is true for every Higgs bundle $\mathcal{V}$ on $X$. Our first result states that it is enough to check the validity of $\mathcal{P}_\mathcal{V}$ for Higgs bundles $\mathcal{V}$ on $X$ whose underlying locally free sheaf has trivial determinant and is semistable on the fibers of $\pi$:

\begin{Prop}
\label{Prop 56}
Assume claim $\mathcal{P}_{\mathcal{V}}$ to be true for  Higgs bundles $\mathcal{V}=(V,\phi)$ on $X$ such that $V$ is fiberwise semistable and has trivial determinant. Then claim $\mathcal{P}_\mathcal{V}$ is true for every Higgs bundle $\mathcal{V}$ on $X$.
\end{Prop}

\begin{proof}
Let $\mathcal{V}=(V,\phi)$ be a rank $r\geqslant 1$ Higgs bundle on $X$ satisfying $\Delta(V)\neq 0$. Let $\mathcal{W}=(W,\psi)=\mathcal{V}\otimes\mathcal{V}^\vee$ (see Section \ref{operations} for the definition of tensor product of Higgs bundles). Then the sheaf $W= V\otimes V^\vee\simeq \mathit{End}(V)$ is locally free or rank $r^2\geqslant 1$ and trivial determinant, and so
$$
\Delta(W)=2r^2\mathrm{c}_2(W)=2r^2\Delta(V)\neq 0.
$$
If the restriction $W_b\simeq V_b\otimes V_b^\vee$ of $W$ to some closed fiber $X_b$ of $\pi$ is unstable, then so is $V_b$, since it is well-known that the slope-semistability of a bundle is equivalent to that of its dual, and the tensor product of two slope-semistable bundles is slope-semistable; analogously, the pull-back $f^\ast V$ is unstable, where $f:C\to X$ is the composition
$$
C\xrightarrow{\nu}X_b\hookrightarrow X,
$$
$\nu:C\to X_b$ being the normalization of $X_b$. But, since $C$ has genus $\leqslant 1$, the instability of $f^\ast V$ implies that of $f^\ast\mathcal{V}$
\cite{bruzzo alessio, garcia-prada}, and claim $\mathcal{P}_\mathcal{V}$ is true. 

Let us then assume that $W$ is fiberwise semistable. By our assumption, claim $\mathcal{P}_\mathcal{W}$ is true, and hence the pull-back $f^\ast\mathcal{W}\simeq f^\ast\mathcal{V}\otimes f^\ast\mathcal{V}^\vee$ is unstable for some curve $C$ and morphism $f:C\to X$; thus, so is $f^\ast\mathcal{V}$ (by the properties of duals and tensor products of semistable \emph{Higgs bundles} analogous to those valid for semistable bundles), showing that claim $\mathcal{P}_\mathcal{V}$ is true in this case too.
\end{proof}

Let us then focus our attention on Higgs bundles $\mathcal{V}=(V,\phi)$ on $X$ such that $V$ has rank $r\geqslant 2$ and trivial (or, more generally, vertical) determinant, and is fiberwise semistable. Such a $V$ has $\Delta(V)=2r\mathrm{c}_2(V)$, and $\mathrm{c}_2(V)\geqslant 0$ (by Proposition \ref{e uguale c2}). Thus, the assumption $\Delta(V)\neq 0$ {in} claim $\mathcal{P}_\mathcal{V}$ is actually equivalent to $\Delta(V)>0$, or to $\mathrm{c}_2(V)>0$. 
 Using the results from the previous section, we are able to prove that  claim $\mathcal{P}_\mathcal{V}$ is true, as soon as $\mathrm{c}_2(V)$ (or $\Delta(V)$) is  big enough, and the spectral cover of $V$ is sufficiently general. We see this as a strong indication that Conjecture \ref{conjecture} (or, at least, a \emph{generic} version of it) is true.

\begin{Prop}
\label{conj rel}
Let $(r,e)$ be a pair of integers satisfying $r\geqslant 2$ and $e\geqslant rd+2g$. Fix a line bundle $\mu$ of degree $e$ on $B$, and let $C$ be a general element of the linear series $\mathbb{P}:=\left|r\Sigma+\pi^\ast\mu\right|$ on $X$. Let $\mathcal{V}=(V,\phi)$ be a Higgs bundle on $X$, whose underlying locally free sheaf $V$ has rank $r$, vertical determinant, $\mathrm{c}_2(V)=e$, and is fiberwise semistable with spectral curve $C_V=C$. Then claim $\mathcal{P}_\mathcal{V}$ is true.
\end{Prop}

\begin{proof}
By Proposition \ref{base point free}, the linear series $\mathbb{P}$ is base-point free. Thus, $C$ can be assumed to be smooth (and, in particular, reduced). By Proposition \ref{sem implies reg}, $V$ is fiberwise regular. Point (ii) of Proposition \ref{uno dei risultati} then allows   us to conclude that the Higgs field $\phi$ is scalar. Moreover, by Theorem \ref{caratterizzazione dei bandoli con...}, $V$ is not curve-semistable. We can then pick a curve $D$, and a morphism $f:D\to X$, such that $f^\ast V$ is unstable; it follows that the Higgs bundle $f^\ast\mathcal{V}$ is also unstable by the remarks of Section \ref{scalar Higgs}, in particular the commutativity of diagram \eqref{scalardiag}. Thus $\mathcal{P}_\mathcal{V}$ holds true, as {claimed}.
\end{proof}

\end{document}